\documentclass[11pt,reqno]{amsart}
\usepackage[
  letterpaper,
  margin=1.15in
]{geometry}               
\usepackage[pagewise]{lineno}
\usepackage{amsmath,amsthm,amssymb,amsfonts,relsize}
\usepackage[english]{babel}
\usepackage{graphicx}
\usepackage{enumitem}
\usepackage{float}
\usepackage{mathtools}
\usepackage{theoremref}
\usepackage{tikz}
\usepackage{pgfplots}
\usetikzlibrary{fillbetween}
\numberwithin{equation}{section}
\newtheorem{theorem}{Theorem}[section]
\newtheorem{remark}[theorem]{Remark}

\newtheorem{lemma}[theorem]{Lemma}

\usepackage{blindtext}
\usepackage{xcolor}

\usepackage{hyperref}
\DeclarePairedDelimiter{\norm}{\lVert}{\rVert}
\newcommand{\R}{\mathbb{R}}

\newcommand{\wh}{\widehat}

\newcommand{\F}{\mathcal{F}}

\newcommand{\eps}{\varepsilon}
\author[D. Sinambela]{Daniel Sinambela}
\address{Department of Mathematics, University of Utah, Asia Campus} 
\email{daniel.sinambela@utah.edu}
\begin{document}
\title[free-surface Reconstruction Error]{Generalized Error Bounds in the Recovery of Solitary Wave Profiles}
\begin{abstract}
We investigate the robustness of Constantin’s explicit reconstruction formula for two-dimensional irrotational solitary water waves. This formula recovers the free-surface profile from the dynamic pressure trace at the bed and depends on both the wave speed and the undisturbed depth. We consider simultaneous perturbations in these three quantities and derive an $L^2$ error estimate for the reconstructed profile. The proof uses the hodograph transform, holomorphic extension arguments, and Paley--Wiener Fourier-decay estimates, yielding stability estimates with sublinear dependence on the perturbation size. We include numerical computations to illustrate the effects of  specifically designed perturbations.
\end{abstract}
\maketitle

\tableofcontents
\section{Introduction}

\indent Reconstructing the free surface of a water wave from pressure measurements is a classical inverse
problem in fluid mechanics. The pressure plays an important role in establishing various qualitative properties of traveling water waves in irrotational flow, e.g. in the description of the particle trajectories beneath the waves \cite{Constantin2012, ConstantineStrauss2010}. For solitary waves, in particular, there have been a plethora of works  devoted to show that the dynamic pressure on the seabed indeed provides useful information for the recovery of the free surface of the wave, see \cite{ClamondConstantin2013,ChenWalsh2018,ChenWalshHur2017}. The simplest recovery approximation formula in this direction is given by the hydrostatic law
\[
\eta(x)\approx \frac{1}{g}\mathfrak{p}(x),
\]
where $g$ is the gravitational acceleration constant and $\mathfrak{p}$ is the dynamic pressure on the bed given by \begin{equation}\label{dynamic pressure trace}
\mathfrak p(x):=P(x,-d)-P_{\mathrm{atm}}-gd.
\end{equation}

A more refined linear theory, derived by Escher and Schl\"urmann \cite{Escher2008}, yields the transfer-function relation,
\[
\F\{\eta\}(k)=\frac{\cosh(kd)}{g}\,\F\{\mathfrak{p}\}(k),
\]
obtained by linearizing the water-wave equations about a flat free surface, where the Fourier transform $\F$ (and its inverse $\F^{-1}$) is taken in the convention
\[
\F\{f\}(k)=\int_{\R} f(x)e^{-ikx}\,dx,
\qquad
\F^{-1}\{\wh f\}(x)=\frac{1}{2\pi}\int_{\R}\wh f(k)e^{ikx}\,dk.
\]
As expected, the linear approximation formulae above are effective only in sufficiently shallow-water regimes. Field-data comparisons also show that linear transfer-function methods can miss high-frequency content and underestimate crest elevations in nonlinear waves \cite{mouragues2019field}. This occurs because the linear theory is unable to capture the dispersive or nonlinear
effects present in the full water-wave problem. It is well understood, for instance in coastal engineering, that nonlinear effects play a major significant role in the problem of reconstructing the free surface for shallow-water waves and for waves propagating in the surf zone, see for instance \cite{bishop1987measuring,tsai2005recovery,bergan1969wave,Bonneton2017}. These shortcomings have prompted efforts to derive reconstruction formula that captures nonlinear effects and remain valid in the solitary-wave setting.

In particular, several authors have obtained nonlinear, nonlocal relations linking the dynamic pressure trace at the bed to the corresponding solitary-wave surface profile, see for instance \cite{Constantin2012,Deconinck2012,Oliveras2012}.
Oliveras et al. \cite{Oliveras2012} derive a formula that encodes an implicit and
nonlocal relation between pressure and surface elevation
\[
\sqrt{\frac{c^2-2g\eta(x)}{1+\eta_x^2(x)}}
=
\mathcal{F}^{-1}
\left\{
\cosh\!\bigl(k(\eta+d)\bigr)\,
\mathcal{F}
\left\{
\sqrt{c^2-2\mathfrak{p}}
\right\}(k)
\right\},
\qquad x\in\mathbb{R}.
\] This fully nonlinear approach was later developed into an operational reconstruction method and tested against laboratory measurements, see \cite{Vasan2017}.

Another well-known surface recovery formula that incorporates the nonlinear effects is given by Constantin \cite{Constantin2012}. In comparison with the formula of Oliveras et al., Constantin's formula provides an explicit
parametric reconstruction for irrotational solitary waves that does not require solving an implicit nonlocal equation for $\eta$, making it possible to recover the
free surface directly from pressure data taken at the flat bed. More precisely, Constantin's formula reads
\begin{equation}\label{eq:constantin_formula}
\left\{
\begin{aligned}
x(q)
&=
q+\int_{-\infty}^{q}
\F^{-1}\!\left[
\cosh(kd)\,
\F\!\left(
\frac{1}{\sqrt{1-2\kappa \mathfrak p}}-1
\right)(k)
\right](s)\,ds,\\
\eta(q)
&=
\F^{-1}\!\left[
\frac{\sinh(kd)}{k}\,
\F\!\left(
\frac{1}{\sqrt{1-2\kappa \mathfrak p}}-1
\right)(k)
\right](q),
\end{aligned}
\right.
\end{equation}
where $\eta=\eta(q)$ is the solitary-wave profile, $d>0$ is the undisturbed depth, and $
\kappa=\frac{1}{c^2},$
with $c>0$ the wave speed. We also briefly mention several related developments beyond the irrotational solitary-wave setting considered here. Recovery methods, in the same spirit as the aforementioned papers, have also been extended to extreme waves, rotational flows, overhanging profiles, and waves with arbitrary bed pressure; see, for instance, \cite{ClamondDidierHenry2020,HenryThomas2018,labarbe2023general} and the references therein.

The availability of such recovery formulas raises a natural and practically important question: how stable are they under perturbations of the input data? In any realistic measurement process, the pressure trace is subject to error, and the physical parameters entering the reconstruction, such as the wave speed and the depth, are not known exactly. Thus, although the pressure trace formally determines the surface profile, one must understand how errors in the data propagate through the reconstruction map. This issue is particularly delicate because the formulas involve Fourier multipliers and analyticity properties, and hence small perturbations in the data need not lead to straightforward estimates in the reconstructed profile.

In her thesis, Bleile \cite{bleile2016} studied the robustness of Constantin's formula under perturbations of the wave speed alone
and showed that the reconstruction error can be bounded in $L^2$, with a sublinear exponent
coming from analyticity and Paley--Wiener decay. Since the bed pressure trace depends, both explicitly and implicitly, on the depth $d$, it is natural to expect that the depth $d$ also plays a role in implicitly dictating the free surface. Hence, practically any noise in the depth is expected to introduce an error in the recovery of the free surface. The goal of the present paper is to extend the
analysis of Bleile \cite{bleile2016} to a more practical perturbation model that takes into account three possible sources of errors: wave speed, pressure trace, and depth. In summary, the present work investigates the stability of \eqref{eq:constantin_formula} when the three aforementioned errors are present.

The paper is organized as follows. In Section 2, we summarize the water-wave setting, introduce the hodograph transform, and re-derive Constantin’s formula. In Section 3, we introduce the perturbed quantities and state the main results. Section 4 contains the proof of the main theorem. Finally, in Section 5, we present numerical results illustrating the free-surface reconstruction errors for particular choices of perturbations.

\section{Setting}\label{sec:setting}
\subsection{Governing Equations}
We consider a two-dimensional solitary wave propagating with speed $c>0$ over a flat bed of depth
$d>0$. The fluid occupies the domain
\[
D:=\{(X,Y): X\in\R,\ -d\le Y\le Z(X-ct)\},
\]
where $Z=Z(X-ct)$ is the traveling-wave profile in moving frame. Passing to a frame moving to the right with speed $c$, we introduce a new coordinate system, namely
\[
x=X-ct,\qquad y=Y.
\]
In the new coordinate system, the flow becomes steady and the fluid domain now reads
\[
\Omega:=\{(x,y): x\in\R,\ -d\le y\le \eta(x)\}.
\]

We assume that the fluid is governed by the steady irrotational incompressible Euler equations 
\begin{equation}\label{eq:euler}
\begin{cases}
u_x+v_y=0,\\
(u-c)u_x+vu_y=-P_x,\\
(u-c)v_x+vv_y=-P_y-g,
\end{cases}
\qquad \text{in }\Omega,
\end{equation}
together with the irrotationality condition
\begin{equation}\label{eq:irrot}
u_y-v_x=0.
\end{equation}
On the boundaries of $\Omega$, we impose standard boundary conditions
\begin{equation}\label{eq:bc}
\begin{cases}
v=0, & y=-d,\\
v=(u-c)\eta_x, & y=\eta(x),\\
P=P_{\mathrm{atm}}, & y=\eta(x),
\end{cases}
\end{equation}
where $P_{\mathrm{atm}}$ is the constant atmospheric pressure.
Since we investigate solitary water waves, we impose the following limiting conditions on the surface profile and the velocity of the wave:
\[
\eta(x)\to 0,\qquad (u-c,v)\to (-c,0)\qquad \text{as }|x|\to\infty.
\]
Additionally, we assume the no-stagnation condition inside the bulk of the fluid, which translates into the inequality
\begin{equation}\label{eq:no-stag}
u-c<0 \qquad \text{in }\Omega.
\end{equation}

\subsection{Hodograph Transform}
The water-wave problem is a free-boundary problem in the sense that the fluid domain is not known a priori: the free surface is itself part of the solution. This feature introduces a substantial additional layer of difficulty, since one must determine simultaneously both the governing flow variables and the geometry of the domain on which they are defined. A standard approach for overcoming this is to employ the Hodograph transform, which maps the unknown fluid region onto a fixed domain by flattening the free surface. In the transformed variables, the problem is reformulated on a prescribed geometry, making it significantly more amenable to analysis. We briefly describe this construction below.

Due to the incompressibility and irrotationality conditions, we can express the velocity field in the following way
\[
u-c= \psi_y, \qquad -v=\psi_x,
\]
where $\psi$ is called the stream function.
Further, there exists a velocity potential $\phi$ satisfying the equations 
\[
\phi_x=u-c, \quad \phi_y=v.
\]
One can easily check that both the stream function $\psi$ and the velocity potential $\phi$ are harmonic.   

Let us now introduce a new complexified variable $z:=x+iy$ and the function $f(z):=(\phi(x,y)+i\psi(x,y))$.
Differentiating $f$ with respect to $z$ yields 
\[
\frac{df}{dz}=\phi_x-i\phi_y=(u-c)-iv.
\]
Via the Cauchy--Riemann equations, it is straightforward to see that $f$ is holomorphic in the complex domain. To fix the domain, we now introduce the hodograph transform  $\mathcal{H}$    taking advantage of the streamfunction and velocity potential. The map $\mathcal{H}$ sends 
 the domain $\Omega$ to a fixed rectangular strip in the following way:
\[
\mathcal{H}:\Omega \to \mathbb{R}\times[-d,0], \qquad \mathcal{H}(x,y):=\left(\frac{-\phi(x,y)}{c},-\frac{\psi(x,y)}{c}\right).
\]
Under the no-stagnation assumption, $\mathcal{H}$ defines an analytic bijection from the domain $\Omega$ to $T:=\mathbb{R} \times [-d,0]$. For the sake of notational convenience, we introduce the following variables to denote the coordinates in $T$:
\[
q=-\frac{\phi(x,y)}{c},\quad  p=-\frac{\psi(x,y)}{c}.\]
Observe that under the mapping $\mathcal{H}$,
the bed $y=-d$ gets mapped to $p=-d$ and the free surface $y=\eta(x)$ gets mapped to $p=0$. We refer the reader to Figure~\ref{fluid domain figure} for the visualization of the original and transformed domains.
\begin{figure}\label{fig1}
\centering
\begin{tikzpicture}[xscale=1.25,yscale=0.65]


\draw [fill=gray,thin,gray] (-4,-1.7) rectangle (4,-1.5);
\draw [thin] (-4,-1.5) -- (4,-1.5);

\draw [domain=-4:4,ultra thick,smooth] 
  plot (\x,{2.5+2.5/(exp(\x)+exp(-\x))});

\node [below] at (1.25,4.5) {\footnotesize $y=\eta(x)$};
\node [below] at (1.25,-0.65) {\footnotesize $y=-d$};

\node [right] at (3.6,0.35) {\footnotesize $h(q,p)=y+d$};
\draw [thick,stealth-stealth] (3.6,-1.35) -- (3.6,2.5);

\node at (0,1) {\footnotesize $\Omega$};

\draw[line width=2pt,->] (0,-1.75) -- (0,-3.25);

\draw [thin] (-4,-3.4) -- (4,-3.4);
\draw [thin] (-4,-7.2) -- (4,-7.2);

\node[below left] at (0.75,2.60) {$y=-d$};

\node[right] at (0.3,-2.2) {$q=-\phi/c$};
\node[right] at (0.3,-2.8) {$p=-\psi/c$};

\node[below left] at (-3.1,-3.5) {$p=0$};
\node[below left] at (-3,-6.3) {$p=-d$};

\end{tikzpicture}
\caption{The Hodograph transform}
\label{fluid domain figure}
\end{figure}
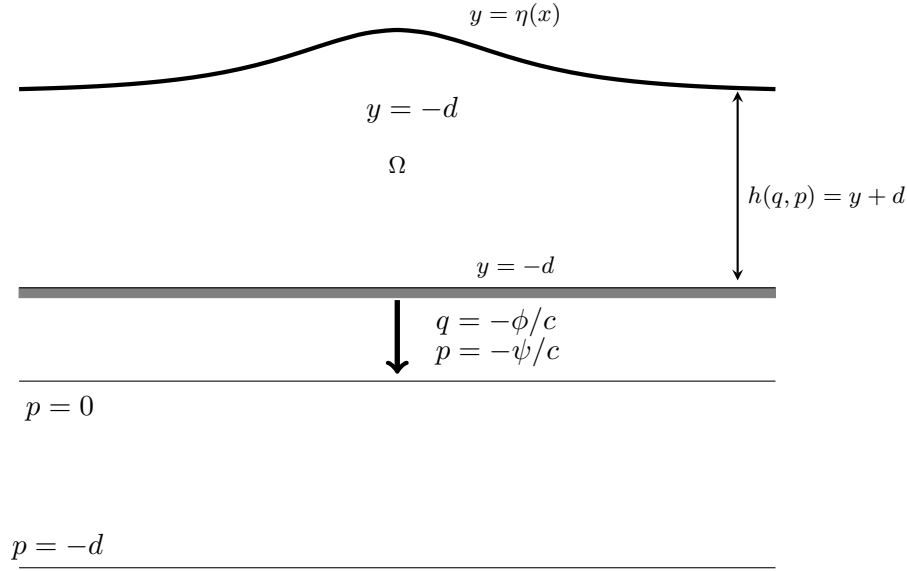
\subsection{Derivation of Constantin’s Formula}
For the sake of completeness, in this subsection we re-derive Constantin's formula. Some of the equations and identities used in the derivation will also be needed in our analysis. Hence, we include them here. We begin by defining the height function $h(q,p)=y+d.$ As a result we obtain the following useful partial derivative formulas,
\[
\begin{aligned}
&\partial_q=h_p\partial_x +h_q\partial_y,\\
&\partial_p=-h_q\partial_x +h_p\partial_y.
\end{aligned}
\]
Employing the above formulas, we then obtain two useful identities
\begin{equation}\label{hqhp}
h_q=\frac{\partial y}{\partial q}=-\frac{\partial x}{\partial p}=-\frac{cv}{(u-c)^2+v^2}, \qquad h_p=\frac{\partial y}{\partial p}=\frac{\partial x}{\partial q}=-\frac{c(u-c)}{(u-c)^2+v^2}.
\end{equation}














Now, we let 
\[
w := (u-c,v)
\]
denote the velocity field in the frame moving at speed $c$.  
The steady Euler equations \eqref{eq:euler} in this frame read
\begin{equation*}
(w \cdot \nabla) w = - \nabla P - g \, \hat{\mathbf{y}}, 
\qquad \nabla \cdot w = 0,
\qquad \nabla \times w = 0,\qquad \hat{\mathbf{y}}=(0,1).
\end{equation*}
Using the irrotationality condition and the vector identity
\begin{equation*}
(w \cdot \nabla) w 
= \nabla \!\left( \tfrac{1}{2} |w|^2 \right) - w \times (\nabla \times w),
\end{equation*}
 the Euler equations reduce to
\begin{equation*}
\nabla \!\left( \tfrac{1}{2} |w|^2 + g y + P \right) = 0,
\end{equation*}
which is equivalent to saying
\begin{equation*}
\tfrac{1}{2} |w|^2 + g y + P = \text{constant } .
\end{equation*}

To determine the constant, we use the far--field behavior: 
as $x \to \pm \infty$, the free surface tends to $y=0$, 
the pressure tends to $P_{\mathrm{atm}}$, and the velocity tends to $(u,v) \to (0,0)$ in the moving frame, 
so that $w \to (-c,0)$.  
Therefore, we obtain
\begin{equation*}
2 \,(P - P_{\mathrm{atm}}) + 2 g y + (u-c)^2 + v^2 = c^2, \qquad \textup{in } \Omega,
\end{equation*}
which is known as Bernoulli's equation. Using this equation together with the fact that 
\begin{equation} \label{velocity squared}
    (u-c)^2+v^2=\frac{c^2}{h_q^2+h_p^2},
\end{equation}
the Bernoulli equation now reads
\[
\frac{c^2}{(h_p^2+h_q^2)}+2(P-P_{\text{atm}})+2g(h-d)=c^2, \qquad \text{in } T.
\]
Evaluating the above equation on the bed $p=-d$ yields
\[
\frac{c^2}{h_p^2(q,-d)}+2(P_b-P_{\text{atm}})-2gd=c^2,
\]
where $P_b=P(x,-d)$ is the pressure measured on the bed. 

In the remaining portion of the article, we will recycle the notation $\mathfrak{p}$ to denote the dynamic pressure on the bed in hodograph domain, $T$. Thus, we can write the above equation in terms of $\mathfrak{p}$ as follows
\[
\frac{c^2}{h_p^2(q,-d)}=c^2-2\mathfrak{p}(q), \quad \textup{for all } q\in \mathbb{R}.
\]
Solving for $h_p(\cdot,-d)$, we get
\begin{equation}
h_p(q,-d)=\frac{1}{\sqrt{1-2\kappa \mathfrak{p}(q)}}, \text{ where} \qquad \kappa=\frac{1}{c^2}.
\label{eq:Neumann-bed}
\end{equation}
\medskip

Furthermore, via the incompressibility and irrotationality assumption, we deduce
\begin{equation}
h_{qq}+h_{pp}=0 \quad \text{in } T.
\label{eq:Laplace}
\end{equation}
We solve \eqref{eq:Laplace} by applying the Fourier transform in $q-$variable. 
Applying $\F$ in $q$ to \eqref{eq:Laplace} gives, for each $k\in\mathbb{R}$,
\begin{equation}
\partial_{pp}\,\widehat{h}(k,p) - k^2\,\widehat{h}(k,p)=0, \qquad -d\le p\le 0.
\label{eq:ODE}
\end{equation}
The general solution of \eqref{eq:ODE} is
\begin{equation}
\widehat{h}(k,p) = A(k)\,\cosh\!\big(k(p+d)\big)\;+\;B(k)\,\sinh\!\big(k(p+d)\big).
\label{eq:Hgeneral}
\end{equation}
Evaluating \eqref{eq:Hgeneral} at $p=-d$ yields
\[
\widehat{h}(k,-d)=A(k)\cosh(0)+B(k)\sinh(0)=A(k)=\F\{h(\cdot,-d)\}(k)=0,
\]
hence
\begin{equation*}
A(k)=0.
\label{eq:Azero}
\end{equation*}
Differentiating \eqref{eq:Hgeneral} in $p$ and evaluating the result at $p=-d$ allow us to infer
\[
\partial_p\widehat{h}(k,p)\big|_{p=-d}=A(k)\,k\,\sinh(0)+B(k)\,k\,\cosh(0)=k\,B(k).
\]
Applying the Fourier transformation to \eqref{eq:Neumann-bed} in $q$ gives
\begin{equation}\label{hp on bed}
\partial_p\widehat{h}(k,-d)=\F\!\left\{\frac{1}{\sqrt{1-2\kappa\mathfrak{p}(q)}}\right\}(k).
\end{equation}
Therefore,
\begin{equation}
B(k)=\frac{1}{k}\,\F\!\left\{\frac{1}{\sqrt{1-2\kappa\mathfrak{p}(q)\,}}\right\}(k)
\qquad (k\neq 0).
\label{eq:Bcoeff}
\end{equation}
Finally, combining \eqref{eq:Azero}--\eqref{eq:Bcoeff} in \eqref{eq:Hgeneral} we obtain, for $k\neq 0$,
\begin{equation}\label{formula h hat}
\widehat{h}(k,p)=\frac{\sinh\!\big(k(p+d)\big)}{k}\;
\F\!\left\{\frac{1}{\sqrt{1-2\kappa\mathfrak{p}(q)\,}}\right\}(k).
\end{equation}

For the zeroth mode, i.e, when $k=0$, note that
\[ 
\widehat{h}(0,p)=(p+d)\,\partial_p\widehat{h}(0,-d).
\]But, since
$\sinh(k(p+d))/k\to (p+d)$ as $k\to 0$, the formula \eqref{formula h hat} extends continuously to $k=0$ by this limit.
Combining both cases gives the desired identity
\begin{equation}
\widehat{h}(k,p)=\frac{\sinh\!\big(k(p+d)\big)}{k}\;
\F\!\left\{\frac{1}{\sqrt{\,1-2\kappa\mathfrak{p}(q)\,}}\right\}(k),\quad k\in\mathbb{R},
\label{eq:211}
\end{equation}
with the $k=0$ value understood in the limiting sense. 
Using \eqref{velocity squared} and \eqref{eq:Neumann-bed}, we can infer that 
\[
c-u(x,-d)=c^2-2\mathfrak{p}.
\]
Via the decay property of $u(x,-d)$ in the far fields, see \cite{Craig1992} and equation \eqref{hqhp}, we can conclude that $\mathfrak{p}$, $\mathfrak{p}_q$ belong to $L^2_q$.

\section{Perturbed Setting and Stability Analysis}
This section records the main statement of the theorem of this work. We begin the section by introducing the perturbed parameters, the perturbations, and assumptions we impose on them. We define
\begin{equation}\label{perturbations}
\begin{aligned}
    \kappa_\varepsilon &:= \kappa + \varepsilon , \\
    \mathfrak{p}_{\delta}(q) &:= \mathfrak{p}(q)+ \delta(q), \\d_\gamma &:= d + \gamma.
\end{aligned}
\end{equation}
 where \( \varepsilon, \delta(q), \gamma \) are all sufficiently small perturbations and $\delta \in L^2_q$. We use \( \eta_{\varepsilon,\delta,\gamma} \) to denote the profile reconstructed from the perturbed parameters above via Constantin's formula \eqref{eq:constantin_formula}. The main theorem below records the error between the surface profiles 
$\eta_{\varepsilon,\delta,\gamma}$ and 
$\eta$ in the $L^2_q$ sense.


\begin{theorem}[Error Bounds]\label{main}
Suppose that the three parameters are measured incorrectly as in \eqref{perturbations} and that $\delta$ admits a holomorphic extension $\delta^*$ to $S^*$. Let $\eta_{\varepsilon,\delta,\gamma}$ be the surface profile  reconstructed from the perturbed parameters. Then, for \( |\varepsilon|, \|\delta\|_{L^2_q}, |\gamma| \ll 1 \), the reconstruction error satisfies
\begin{equation} \label{eq:main_bound}
\| \eta_{\varepsilon,\delta,\gamma} - \eta \|_{L^2_q} \leq \left(|\gamma| |\kappa|^{\frac{(\sigma-\gamma)}{d+\sigma}}+(d+\gamma) |\varepsilon|^{\frac{(\sigma-\gamma)}{d+\sigma}} \right) \mathcal{A}_1+ \left((d+\gamma)\norm{\delta}^{\frac{(\sigma-\gamma)}{d+\sigma}}_{L^2_q}|\kappa|^{\frac{(\sigma-\gamma)}{d+\sigma}}\right)\mathcal{A}_2,
\end{equation}
where the exponent $\frac{(\sigma-\gamma)}{\sigma + d} \in (0,1)$. Here, \( \mathcal{A}_1,  \mathcal{A}_2>0 \) depend on \(\mathfrak{p}, \gamma, d, \sigma \) and $\mathcal{A}_1= \norm{\mathfrak{p}}^{\frac{(\sigma-\gamma)}{d+\sigma}}_{L^2_q} \mathcal{A}_2 $.
\end{theorem}


\begin{remark}
We make several comments about the result.
\begin{enumerate}
    \item If the perturbations $\delta$ and $\gamma$ are set to zero, then the estimate reduces to the corresponding stability bound obtained in \cite{bleile2016}.

    \item The perturbation $\delta$ is allowed to depend on the variable $q$, and our assumptions require it to be sufficiently small in the $L^2$-sense. Moreover, to apply the Paley--Wiener theorem in the form needed for the proof, we must assume that $\delta$ admits a holomorphic extension $\delta^*$ to the relevant complex domain.

    \item Although the resulting estimate is necessarily more involved than the one in \cite{bleile2016}, the reconstruction error still depends sublinearly on the size of the perturbations. In particular, the theorem shows that the stability mechanism identified in the single-parameter setting persists in this more general perturbative regime.
\end{enumerate}
\end{remark}
We postpone the proof of Theorem~\ref{main} and present it in Section~\ref{main proof}. Below, we present a sequence of auxiliary theorem and lemmas that we use to complete the proof of the main theorem. We start with the classical Paley--Wiener theorem.

\begin{theorem}[Paley--Wiener]\label{paley--wiener}
Let $F(q+ip)$ be a function holomorphic in the horizontal strip
\[
S := \{\, q+ip \in \mathbb{C} : p \in [-a,b] \,\}, \qquad a,b>0,
\]
and assume that for each fixed $p \in [-a,b]$,
$
F(\cdot + ip) \in L^2_q$.
Then 
\[
\int_{\mathbb{R}} |\widehat{F}(k,p)|^2 e^{2b|k|} \, dk < \infty,
\qquad
\int_{\mathbb{R}} |\widehat{F}(k,p)|^2 e^{-2a|k|} \, dk < \infty.
\]

In particular, there exist constants $C,\sigma>0$ such that
\[
|\widehat{F}(k,p)| \le C e^{-\sigma |k|}, \qquad k \in \mathbb{R}.
\]
\end{theorem}

\begin{proof}
This is a classical form of the Paley--Wiener theorem for holomorphic functions on a horizontal strip. Since the result is standard, we omit the proof and refer the reader, for instance, to \cite{stein2011fourier}.
\end{proof}
The next lemma states that the dynamical pressure on the bed admits a holomorphic extension in a complex strip.

\begin{lemma}[Holomorphic extension of the bed pressure]\label{holo ext p}
Let $\mathfrak p(q)$ denote the dynamic pressure trace on the bed, expressed in
hodograph coordinates. Then there exists a function $\mathfrak P^*$ holomorphic
in the hodograph strip
\[
S^*:=\{q+ip:q\in\mathbb R,\; -2d<p<0\}
\]
such that
\[
\mathfrak P^*(q-id)=\mathfrak p(q).
\]
Moreover, for each fixed $p\in[-2d,0]$,
\[
\mathfrak P^*(\cdot+ip)\in L^2_q .
\]
\end{lemma}

\begin{proof}
We first construct the extension in the physical variables and then transfer it
to the hodograph variables. Recall that the hodograph map is given by
\[
\mathcal H(x,y)
=
\left(-\frac{\phi(x,y)}{c},-\frac{\psi(x,y)}{c}\right)
=(q,p).
\]

Since $\mathcal{H}$ is a holomorphic bijection, it admits a holomorphic inverse
\[
\mathcal H^{-1}(q,p)=(x(q,p),y(q,p)).
\]
Let
\[
w(z)=(u(x,y)-c)-iv(x,y),\qquad \text{with } z=x+iy.
\] Using the fact that the flow is incompressible and
irrotational, $w$ is holomorphic in the fluid domain $\Omega$. We now reflect $w$ across the flat bed. The boundary condition on the bed gives
\[
v(x,-d)=0.
\]
Hence, after translating the bed to the line $Y=0$ by setting $Y=y+d$, the imaginary
part of $w$ vanishes on $Y=0$. By the Schwarz reflection principle, $w$ admits a
holomorphic extension across the bed. Equivalently, defining
\[
\widetilde w(x+iY)
=
\begin{cases}
w(x+i(Y-d)), & Y\geq 0,\\[4pt]
\overline{w(x-iY-id)}, & Y<0,
\end{cases}
\]
gives a holomorphic extension of the complex velocity to the reflected physical
strip below the bed.

Next, by Bernoulli's equation,
\[
2(P-P_{\mathrm{atm}})+2gy+|w|^2=c^2 .
\]
On the bed $y=-d$, the dynamic pressure is
\[
\mathfrak p(x)
=
P(x,-d)-P_{\mathrm{atm}}-gd.
\]
Since $v(x,-d)=0$, this gives
\[
2\mathfrak p(x)
=
c^2-(u(x,-d)-c)^2.
\]
Thus the bed pressure trace is determined by the square of the complex velocity on
the bed. Therefore the function
\[
\widetilde{\mathfrak P}(z)
:=
\frac12\Bigl(c^2-\widetilde w(z)^2\Bigr)
\]
is holomorphic in the reflected physical strip and agrees with the dynamic pressure
trace on the bed.

Finally, we express this holomorphic extension in hodograph variables. Define
\[
\mathfrak P^*(q+ip)
:=
\widetilde{\mathfrak P}\bigl(\mathcal H^{-1}(q,p)\bigr).
\]
Since both $\widetilde{\mathfrak P}$ and $\mathcal H^{-1}$ are holomorphic, their
composition is holomorphic in the hodograph strip. Moreover, because the physical
bed $y=-d$ is mapped by $\mathcal H$ to the line $p=-d$, we obtain
\[
\mathfrak P^*(q-id)=\mathfrak p(q).
\]

It remains to justify the $L^2_q$ condition on horizontal slices. The solitary-wave
decay assumptions imply that the velocity perturbation decays as $|q|\to\infty$.
Consequently the dynamic pressure trace belongs to $L^2_q$. The same decay holds
on each fixed horizontal line in the reflected hodograph strip, because
$\mathfrak P^*$ is obtained by analytic continuation of the same decaying velocity
field. Hence
\[
\mathfrak P^*(\cdot+ip)\in L^2_q
\qquad\text{for each }p\in[-2d,0].
\]
This proves the lemma.
\end{proof}

Next,  using Lemma~\ref{holo ext p} we show that in Lemma~\ref{holomorphic} 
\begin{equation}\label{Fdeltastar}
F_{\delta^*}(q+ip ):=\frac{1}{\sqrt{1-2\kappa_\varepsilon\bigl(\mathfrak{P}^*(q+ip )+\delta^*(q+ip)\bigr)}},
\end{equation}
is holomorphic on $S^*$ where $\delta^*$ is the holomorphic extension of $\delta$ on $S^*$.
\begin{lemma}\label{holomorphic}
Let $\mathfrak{P}^*$ be the holomorphic extension of $\mathfrak{p}$ from Lemma~\ref{holo ext p} and $\delta^*$ be the holomorphic extension of perturbation $\delta$ to $S^*$.
Define
\[
R_{\delta^*}(z):=1-2\kappa_\varepsilon\bigl(\mathfrak{P}^*(z)+\delta^*(z)\bigr),
\]
and $F_{\delta^*}$ as in \eqref{Fdeltastar}.
For sufficiently small $\varepsilon$, and small $\delta^*$ in $L^\infty(S^*)$, 
$F_{\delta^*}$ is holomorphic on
$S^*$.
\end{lemma}

\begin{proof}
Since $\mathfrak{P}^*$ is holomorphic on $S^*$ by Lemma~\ref{holo ext p} and $\delta^*$ is also assumed
holomorphic on $S^*$, it follows that
$R_{\delta^*}$
is holomorphic on $S^*$. When $\delta^*\equiv0$, the quantity $R_0(z)=1-2\kappa_\varepsilon \mathfrak{P}^*(z)$ is holomorphic. Moreover, by \cite[Lemma 3.2]{bleile2016}, for sufficiently small $\varepsilon$, the function $R_0$
does not intersect the branch cut $(-\infty,0]$ on $S^*$. Since $R_0(S^*)$ is separated
from $(-\infty,0]$, then there exists $m_0>0$ such that
\[
\operatorname{dist}\bigl(R_0(z),(-\infty,0]\bigr)\ge m_0
\qquad \text{for all } z\in S^*.
\]

Now choose $\delta^*$ so that
\[
2|\kappa_\varepsilon|\,\|\delta^*\|_{L^\infty(S^*)}<m_0.
\]
Then, for every $z\in S^*$,
\[
|R_{\delta^*}(z)-R_0(z)|
=
2|\kappa_\varepsilon|\,|\delta^*(z)|
\le
2|\kappa_\varepsilon|\,\|\delta^*\|_{L^\infty(S^*)}
<
m_0.
\]
Therefore,
\[
\operatorname{dist}\bigl(R_{\delta^*}(z),(-\infty,0]\bigr)
\ge
\operatorname{dist}\bigl(R_0(z),(-\infty,0]\bigr)-|R_{\delta^*}(z)-R_0(z)|
>
0
\]
for all $z\in S^*$. Hence $R_{\delta^*}(S^*)\subset \mathbb{C}\setminus(-\infty,0]$.

Since the principal branch of the square root is holomorphic on
$\mathbb{C}\setminus(-\infty,0]$, the map
\[
w \mapsto w^{-1/2}
\]
is holomorphic on the range of $R_{\delta^*}$. Consequently,
\[
F_{\delta^*}(z)=R_{\delta^*}(z)^{-1/2}
\]
is holomorphic on $S^*$.
\end{proof}

The last auxiliary lemma below records the pointwise bound on the difference $\wh{\,g_{\eps,\delta}\,}(k)- \wh{\,g_{0,0}\,}(k)$. 

\begin{lemma}[Decay of \(\wh{g_{\eps,\delta}}\) - \(\wh{g_{0,0}}\)]
Under the assumptions used in Theorem~\ref{main}, there exist \(\sigma>0\) and \(C>0\) such that
\(\wh{g_{\eps,\delta}}\in L^{2}_k\) and
\[
\big|\;\wh{\,g_{\eps,\delta}\,}(k)- \wh{\,g_{0,0}\,}(k)\;\big|\;\le\;C\,e^{-(d+\sigma)\,|k|}\quad(k\in\mathbb{R}),
\]
where 
\[
g_{\eps,\delta}(q)=\frac{1}{\sqrt{1-2(\kappa+\eps)(\mathfrak{p}_{\delta}(q))}}-1, \qquad \qquad g_{0,0}(q)=\frac{1}{\sqrt{1-2\kappa \mathfrak{p}(q)}}-1.
\]
\end{lemma}

\begin{proof}
Notice that $g_{\eps,\delta}$ and $g_{0,0}$ are \(L^{2}_q\) functions provided \(\mathfrak{p},\delta\in L^{2}_q\) and the radicands stay away
from the nonpositive real axis. Using the identity
\[
\frac{1}{\sqrt{1-x}}-1 \;=\; \frac{x}{\sqrt{1-x}\,\bigl(1+\sqrt{1-x}\bigr)},
\]
with \(x=2(\kappa+\eps)(\mathfrak{p}_{\delta})\) or \(x=2\kappa \mathfrak{p}\), we can write
\begin{equation}\label{splitting}
\begin{aligned}
    \frac{1}{\sqrt{1-2(\kappa+\eps)(\mathfrak{p}_{\delta})}} 
- \frac{1}{\sqrt{1-2\kappa \mathfrak{p}}}&=\frac{1}{\sqrt{1-2(\kappa+\eps)(\mathfrak{p}_{\delta})}}-1
- \left(\frac{1}{\sqrt{1-2\kappa \mathfrak{p}}}-1\right)\\&=\frac{2(\kappa+\eps)(\mathfrak{p}_{\delta})}{\sqrt{1-2(\kappa+\eps)(\mathfrak{p}_{\delta})}(1+\sqrt{1-2(\kappa+\eps)(\mathfrak{p}_{\delta})})}\\&\qquad \qquad + \frac{-2\kappa \mathfrak{p}}{\sqrt{1-2\kappa \mathfrak{p}}(1+\sqrt{1-2\kappa \mathfrak{p}})}.
\end{aligned}
\end{equation}
Recalling the fact that $\mathfrak{p}$ and its derivative are square integrable in $q$, we can conclude that both terms on the right-hand side of \eqref{splitting} are also square integrable. Hence, $
\|g_{\eps,\delta}-g_{0,0}\|_{L^{2}_q}< \infty$. Moreover, via the no-stagnation assumption and the definition of $h_p$ in \eqref{hqhp}, it is straightforward to see that $h_p>0$. Notice also that $h\geq 0$ and $h\in W^{1,\infty}$. All these facts combined allows us to infer that there exist some constants $M$ and $m$ such that
\[
\norm{h_p}_{L^\infty}=M\geq 1, \qquad m=\inf h_p.
\]
Hence, recalling \eqref{hp on bed}, on the bed we also obtain
\[
\frac{1}{M^2}\leq 1-2\kappa \mathfrak{p}\leq \frac{1}{m^2}.
\]
As a consequence, we arrive at the following $L^2_q$ estimate
\begin{equation}\label{g00 bound}
    \begin{aligned}
        \norm{g_{0,0}}_{L^{2}_q}=\norm{\frac{2\kappa \mathfrak{p}}{\sqrt{1 - 2\kappa \mathfrak{p}}
        (1+\sqrt{1 - 2\kappa \mathfrak{p}})}}_{L^{2}_q}\leq
        2\kappa M \norm{\mathfrak{p}}_{L^{2}_q}\leq
        2^{3/2}\kappa M^{\frac{3}{2}} \norm{\mathfrak{p}}_{L^{2}_q}.
    \end{aligned}
\end{equation}

Furthermore, for sufficiently small $\varepsilon$, the perturbed quantity $1-2(\kappa+\varepsilon) \mathfrak{p}$ can also be estimated from above and below using the constants $A_1$ and $a_1$ as follows
\[ 
A_1 \leq 1-2(\kappa+\varepsilon) \mathfrak{p} \leq a_1.
\]
Finally, perturbing $\mathfrak{p}$ sufficiently small allows us to maintain a similar estimate where
\[
A_2 \leq 1-2(\kappa+\varepsilon) \mathfrak{p}_{\delta} \leq a_2.
\]
for some constants $A_2$ and $a_2$.
Hence, we obtain
\[
\begin{aligned}
\left|g_{\eps,\delta}\right|&=\left|\frac{2(\kappa+\eps)(\mathfrak{p}_{\delta})}{\sqrt{1-2(\kappa+\eps)(\mathfrak{p}_{\delta})}(1+\sqrt{1-2(\kappa+\eps)(\mathfrak{p}_{\delta})})}\right|\\&\lesssim C(\kappa, \varepsilon,a_2,A_2)\left|\mathfrak{p}+\delta\right|.
\end{aligned}
\]

Now, we define on \(S^{*}\) the following difference
\[
G(q+ip)
:= 
\frac{1}{\sqrt{1-2(\kappa+\eps)\bigl(\mathfrak{P}^{*}_{\delta^*}(q+ip)\bigr)}}
-
\frac{1}{\sqrt{1-2\kappa\,\mathfrak{P}^{*}(q+ip)}},
\]
the extension of $g_{\varepsilon,\delta}-g_{0,0}$, where $\mathfrak{P}^*$ is the holomorphic extension of $\mathfrak{p}$ as in Lemma~\ref{holo ext p}.
Thus, via Lemma~\ref{holomorphic}, \(G\) is holomorphic on \(S^{*}\). 
Moreover, for fixed \(p\in[-2d,0]\), the same algebra used in \eqref{splitting} leads to
\(\,G(\cdot+ip)\in L^{2}_q\), since \(\mathfrak{P}^{*}(\cdot+ip)\in L^{2}_q\) 
and the denominators are uniformly bounded away from \(0\) on \(S^{*}\). Hence, $G$ satisfies the assumptions of Theorem~\ref{paley--wiener}.

Thus, Theorem~\ref{paley--wiener} allows us to infer that
\(G\) admits a Fourier transform in the sense of functions and enjoys two-sided 
exponential decay. More precisely, there exists \(\sigma>0\) (coming from any extra margin of analyticity beyond width \(2d\)) and \(C>0\) such that
\[
\int_{\mathbb{R}} |\widehat{G}(k)|^2 e^{2|k|(d+\sigma)} \;dk <\infty,
\]
and 
\[
\big|\;\wh{\,G(\cdot-id)\,}(k)\;\big|
\;\lesssim 
\,e^{-(d+\sigma)\,|k|},\quad \text{for all }k\in\mathbb{R}.
\]
But \(G(\cdot-id)=g_{\eps,\delta}-g_{0,0}\), so the claimed bound follows.
\end{proof}
\section{Proof of the main theorem}\label{main proof}

This section records the proof of the main Theorem~\ref{main}. Recalling that through \eqref{eq:constantin_formula} we have
\begin{equation}
\eta_{\varepsilon,\delta,\gamma}(q) - \eta(q) = \mathcal{F}^{-1}\left\{ \frac{\sinh(k(d + \gamma))}{k} \widehat{g}_{\varepsilon,\delta}(k) - \frac{\sinh(kd)}{k} \widehat{g}_{0,0}(k) \right\}.
\end{equation}
Using the triangle inequality, one can split the difference above into two parts,
\begin{equation}
    \begin{aligned}
\| \eta_{\varepsilon,\delta,\gamma} - \eta \|^2_{L^2_q} &\leq \left\| \mathcal{F}^{-1}\left[ \frac{\sinh(k(d + \gamma)) - \sinh(kd)}{k} \widehat{g}_{0,0}(k) \right] \right\|^2_{L^2_q} \\
&\quad + \left\| \mathcal{F}^{-1}\left[ \frac{\sinh(k(d + \gamma))}{k} (\widehat{g}_{\varepsilon,\delta}(k) - \widehat{g}_{0,0}(k)) \right] \right\|^2_{L^2_q}
\\&=: \textbf{\textit{I}}+\textbf{\textit{II}}.
\end{aligned}
\end{equation}

\textbf{\textit{Estimate of I}:}
The first term is estimated using analyticity and the bound (fixed $k$):
\[ |\sinh(k(d+\gamma)) - \sinh(kd)| \leq |\gamma| |k| \cosh(k(d + \theta \gamma)),\]
for some $\theta \in (0,1)$. Hence, we obtain the following estimate
\[
|\sinh(k(d+\gamma)) - \sinh(kd)| \leq  |\gamma| |k| \cosh(|k|\max{\{d,d+ \gamma\}})\lesssim |\gamma||k|e^{|k|(d+|\gamma|)}.
\]

Since \( g_{0,0}(q) \in L^2_q \) and is analytic in a strip of width  \( \sigma \), by the Paley--Wiener theorem we can infer that \( |\widehat{g}_{0,0}(k)| \lesssim e^{-|k|(d + \sigma)} \). Thus, we obtain the following upper bound for $\textbf{\textit{I}}$:
\[
\textbf{\textit{I}} \lesssim  |\gamma|^2 \norm{ e^{|k|(d + \gamma)}\widehat{g}_{0,0}(k)}^2_{L^2_k}.
\]
Furthermore, let us split up the upper bound of $\textbf{\textit{I}}$ above into two sub-integrals, namely
\begin{equation}
\begin{aligned}
  |\gamma|^2\left(\textbf{\textit{I}}_1+\textbf{\textit{I}}_2\right):&= |\gamma|^2\int_{B_{R}(0)}|\widehat{g}_{0,0}(k)|^2 e^{2|k|(d + \gamma+\sigma/2)}e^{-|k|\sigma}\;dk\\&\qquad \qquad \qquad + |\gamma|^2\int_{(B_{R}(0))^c}|\widehat{g}_{0,0}(k)|^2 e^{2|k|(d + \gamma+\sigma/2)}e^{-|k|\sigma}\;dk \\&
  \leq |\gamma|^2e^{2R(d+\gamma)}\int_{B_{R}(0)}|\widehat{g}_{0,0}(k)|^2\;dk+|\gamma|^2e^{-R\sigma}\int_{(B_{R}(0))^c}|\widehat{g}_{0,0}(k)|^2 e^{2|k|(d + \gamma+\sigma/2)}\;dk\\&
  \leq |\gamma|^2e^{2R(d+\gamma)}\int_{B_{R}(0)}|\widehat{g}_{0,0}(k)|^2\;dk+|\gamma|^2C^2e^{-R\sigma}\int_{(B_{R}(0))^c}e^{-2|k|(d +\sigma)} e^{2|k|(d + \gamma+\sigma/2)}\;dk\\& \leq |\gamma|^2\left(e^{2R(d+\gamma)}\norm{\widehat{g}_{0,0}(k)}^2_{L^2_k}+\frac{C^2 e^{2R(\gamma-\sigma)}}{\sigma-2\gamma}\right),
  \end{aligned}
\end{equation}
for $|\gamma| <\frac{\sigma}{2}$ chosen sufficiently small. Upon minimizing the right quantity above, one can check the optimal bound is attained when 
\[
R=\frac{1}{2(d+\sigma)} \ln\left({\frac{C^2(\sigma-\gamma)}{(\sigma-2\gamma)(d+\gamma) \norm{g_{0,0}}_{L^2_q}^2}}\right).
\]

\textbf{\textit{Estimate of II}:}
We define
\begin{align*}
\Gamma(q; \kappa+\varepsilon, \mathfrak{p}_\delta) &:= \frac{1}{\sqrt{1 - 2(\kappa + \varepsilon)(\mathfrak{p}_\delta)}}.
\end{align*}
By the Taylor expansion, we get
\[
\Gamma(q;\kappa+\varepsilon,\mathfrak p+\delta)
-
\Gamma(q;\kappa,\mathfrak p)
=
\varepsilon\,\partial_\kappa\Gamma(q;\kappa,\mathfrak p)
+
\delta\,\partial_{\mathfrak p}\Gamma(q;\kappa,\mathfrak p)
+
\text{higher-order terms},\]
where we have
\[
\begin{aligned}
\frac{\partial \Gamma}{\partial \kappa} &= \frac{\mathfrak{p}}{(1 - 2\kappa \mathfrak{p})^{3/2}}, \quad
\frac{\partial \Gamma}{\partial \mathfrak{p}} = \frac{\kappa}{(1 - 2\kappa \mathfrak{p})^{3/2}}.
\end{aligned}
\]

Hence,
\begin{equation}\label{diff}
\begin{aligned}
\norm{g_{\varepsilon,\delta}(q)-g_{0,0}(q)}_{L^2_q}&=\norm{\Gamma(q; \kappa+\varepsilon, \mathfrak{p}_\delta)-\Gamma(q; \kappa, \mathfrak{p})}_{L^2_q}\\&\lesssim \left\|\varepsilon \frac{\partial \Gamma}{\partial \kappa}+\delta \frac{\partial \Gamma}{\partial \mathfrak{p}}\right\|_{L^2_q}\\&=\left\|\frac{\varepsilon \mathfrak{p}}{(1 - 2\kappa \mathfrak{p})^{3/2}} + \frac{\delta \kappa}{(1 - 2\kappa \mathfrak{p})^{3/2}}\right\|_{L^2_q}\\&
\leq2\sqrt{2} M^{3} (\norm{\mathfrak{p}}_{L^2_q}|\varepsilon|+\norm{\delta}_{L^2_q}|\kappa|)
\end{aligned}
\end{equation}
Now, using the fact that  for all $k \in \mathbb{R}$,
$\frac{\sinh{k(d+\gamma)}}{k}\leq (d+\gamma)e^{|k|(d+\gamma)},$ we therefore obtain 
\begin{equation}
    \textbf{\textit{II}} \leq (d+\gamma)^2\left\|\left[ e^{|k|(d+\gamma)}(\widehat{g}_{\varepsilon,\delta}(k) - \widehat{g}_{0,0}(k)) \right] \right\|^2_{L^2_k}.
\end{equation}
As before, we split $\textbf{\textit{II}}$, 
\begin{equation}
\begin{aligned}
(d+\gamma)^2\left(\textbf{\textit{II}}_1+\textbf{\textit{II}}_2\right)&\leq (d+\gamma)^2 \int_{B_{R}(0)}|\widehat{g}_{\varepsilon,\delta}(k) - \widehat{g}_{0,0}(k)|^2 e^{2|k|(d + \gamma+\sigma/2)}e^{-|k|\sigma}\;dk \\&\qquad + (d+\gamma)^2\int_{(B_{R}(0))^c}|\widehat{g}_{\varepsilon,\delta}(k) - \widehat{g}_{0,0}(k)|^2  e^{2|k|(d + \gamma+\sigma/2)}e^{-|k|\sigma}\;dk\\&
\leq  (d+\gamma)^2 e^{2R(d+\gamma)} \left\|\widehat{g}_{\varepsilon,\delta}(k) - \widehat{g}_{0,0}(k)\right\|^2_{L^2_k}\\&\qquad \qquad  + (d+\gamma)^2 C^2 e^{-R\sigma}\int_{(B_{R}(0))^c} e^{-2|k|(d+\sigma)} e^{2|k|(d + \gamma+\sigma/2)}\;dk\\&
\leq (d+\gamma)^2 \left(e^{2R(d+\gamma)}\norm{\widehat{g}_{\varepsilon,\delta}(k) - \widehat{g}_{0,0}(k)}^2_{L^2_k}+ \frac{C^2 e^{2R(\gamma-\sigma)}}{\sigma-2\gamma}\right),
\end{aligned}
\end{equation}
for $\gamma <\frac{\sigma}{2}$ chosen sufficiently small. Upon minimizing the right quantity above, one can check the optimal bound is attained when 
\[
R=\frac{1}{2(d+\sigma)} \ln\left({\frac{C^2(\sigma-\gamma)}{(\sigma-2\gamma)(d+\gamma) \norm{\widehat{g}_{\varepsilon,\delta}(k) - \widehat{g}_{0,0}(k)}^2_{L^2_k}}}\right).
\]
Upon gathering all the estimates we have, we obtain
\begin{equation} 
    \begin{aligned}
\| \eta_{\varepsilon,\delta,\gamma} - \eta \|^2_{L^2_q} &\leq|\gamma|^2\left(e^{2R(d+\gamma)}\norm{\widehat{g}_{0,0}(k)}^2_{L^2_k}+\frac{C^2 e^{2R(\gamma-\sigma)}}{\sigma-2\gamma}\right) \\&\qquad \qquad+(d+\gamma)^2 \left(e^{2R(d+\gamma)}\norm{\widehat{g}_{\varepsilon,\delta}(k) - \widehat{g}_{0,0}(k)}^2_{L^2_k}+ \frac{C^2 e^{2R(\gamma-\sigma)}}{\sigma-2\gamma}\right)\\&\leq \left(|\gamma|^2 \norm{\widehat{g}_{0,0}(k)}^{\frac{2(\sigma-\gamma)}{d+\sigma}}_{L^2_k}
+(d+\gamma)^2 \norm{\widehat{g}_{\varepsilon,\delta}(k) - \widehat{g}_{0,0}(k)}^{\frac{2(\sigma-\gamma)}{d+\sigma}}_{L^2_k} \right) \\&\qquad \qquad \times\left(\frac{C^2(\sigma-\gamma)}{(\sigma-2\gamma)(d+\gamma)}+\frac{C^2}{\sigma-2\gamma} \left(\frac{C^2(\sigma-\gamma)}{(\sigma-2\gamma)(d+\gamma)}\right)^{\frac{\gamma-\sigma}{d+\sigma}} \right).
\end{aligned}
\end{equation}
Now, using \eqref{g00 bound} and \eqref{diff}, we obtain the following estimate
\begin{equation}
    \begin{aligned}
        \| \eta_{\varepsilon,\delta,\gamma} - \eta \|^2_{L^2_q} &\lesssim \left(|\gamma|^2 |\kappa|^{\frac{2(\sigma-\gamma)}{d+\sigma}} \norm{\mathfrak{p}}^{\frac{2(\sigma-\gamma)}{d+\sigma}}_{L^2_q}+(d+\gamma)^2 \left(\norm{\mathfrak{p}}^{\frac{2(\sigma-\gamma)}{d+\sigma}}_{L^2_q} |\varepsilon|^{\frac{2(\sigma-\gamma)}{d+\sigma}}+\norm{\delta}^{\frac{2(\sigma-\gamma)}{d+\sigma}}_{L^2_q}|\kappa|^{\frac{2(\sigma-\gamma)}{d+\sigma}}\right)\right)
    \end{aligned}
\end{equation}

Combining both contributions proves the Theorem~\ref{main}. \qedhere
\section{Numerical Computations}

In this section, we investigate the sensitivity of the recovery formula under combinations of perturbations in the wave speed, pressure trace at the bed, and fluid depth. Our goal is to complement the theoretical estimates by examining the behavior of the reconstruction error in practice.
In order to numerically measure the size of the deviation of $\eta_{\varepsilon,\delta,\gamma}$ from $\eta$, we introduce the following norm quantity
\begin{equation}\label{Error}
  E(\varepsilon,\delta,\gamma)
=
\|\eta_{\varepsilon,\delta,\gamma} - \eta\|_{L^2_q(\mathbb{R})}.  
\end{equation}

Following the approach in \cite{bleile2016}, we prescribe the auxiliary function
\[
g_{0,0}(q) = e^{-q^2/2},
\]
and solving for $\mathfrak{p}(q)$. With $c = 2$, so that $\kappa = 1/4$, and choosing $d = 1$, this yields
\[
\mathfrak{p}(q)
=
\frac{2\bigl(2 + e^{-q^2/2}\bigr)}
{2 + e^{-q^2/2} + e^{q^2/2}}.
\]
This function is smooth, rapidly decaying, and belongs to $L^2_q$.

To perturb the pressure trace on the bed, we consider the perturbation
\begin{equation}\label{delta}
    \delta(q) = a \, e^{-q^2/(2\theta^2)},
\end{equation}
where $a > 0$ is the amplitude and $\theta > 0$ controls the width of $\delta$. Notice that $\delta$ is entire, hence it admits a holomorphic extension $\delta^*$ to $S^*$ given by $\delta^*(z)=\delta(z)$. Additionally, in order to ensure that the reconstruction formula remains well-defined, we require
\[
1 - 2(\kappa + \varepsilon)\bigl(\mathfrak{p}(q) + \delta(q)\bigr) > 0
\]
for all $q \in \mathbb{R}$.
All perturbations are chosen sufficiently small so that this condition holds.

\subsection{Wave Speed Perturbation}
In this subsection, we consider the simplest perturbative setting in which only the wave-speed parameter is perturbed, while the pressure trace and the depth remain fixed. More precisely, we set $\delta = 0$ and $\gamma = 0$ and allow only $\kappa_\varepsilon=\kappa+\varepsilon
$
to vary, where $\kappa=\frac{1}{c^2}$.
 Therefore, 
\[
g_{\varepsilon,0}(q)
=
\frac{1}{\sqrt{1-2(\kappa+\varepsilon)\mathfrak{p}(q)}}-1.
\]
The perturbed surface profile is then given by
\[
\eta_{\varepsilon,0,0}(q)
=
\mathcal F^{-1}
\left[
\frac{\sinh(kd)}{k}\,\widehat g_{\varepsilon,0}(k)
\right](q).
\]

Since the depth is not perturbed in this scenario, the Fourier multiplier remains $
\dfrac{\sinh(kd)}{k}$
with \(d=1\).
To ensure that the reconstruction formula is well-defined, the radicand must remain positive:
\[
1-2(\kappa+\varepsilon)\mathfrak{p}(q)>0
\qquad
\text{for all } q\in\mathbb R.
\]
Because the chosen pressure satisfies $
\max_{q\in\mathbb R} \mathfrak{p}(q)=\frac{3}{2}$,
and since \(\kappa=1/4\), it is enough to require $1-2\Bigl(\frac14+\varepsilon\Bigr)\frac32>0$.
This simplifies to $\varepsilon<\frac{1}{12}$.
In our numerical codes, we choose $\varepsilon$ in the following interval $
\varepsilon\in\left(-\frac{1}{12},\frac{1}{12}\right).
$

The theoretical analysis predicts sublinear growth of the form
\[
E(\varepsilon,0,0)\sim |\varepsilon|^\alpha
\qquad \text{for some } \alpha\in(0,1),
\]
with the exponent arising from analyticity and Paley--Wiener decay estimates. Numerically, we obtain the graph in Figure~\ref{wavespeederror} suggesting the relationship between the size of wave speed perturbation $\varepsilon$ with the size of the free-surface reconstruction error.
\begin{figure}[H]\label{wave speed error}
\centering
    \begin{tikzpicture}
    \node[inner sep=0pt] at (5,0) {\includegraphics[width=9.5cm]{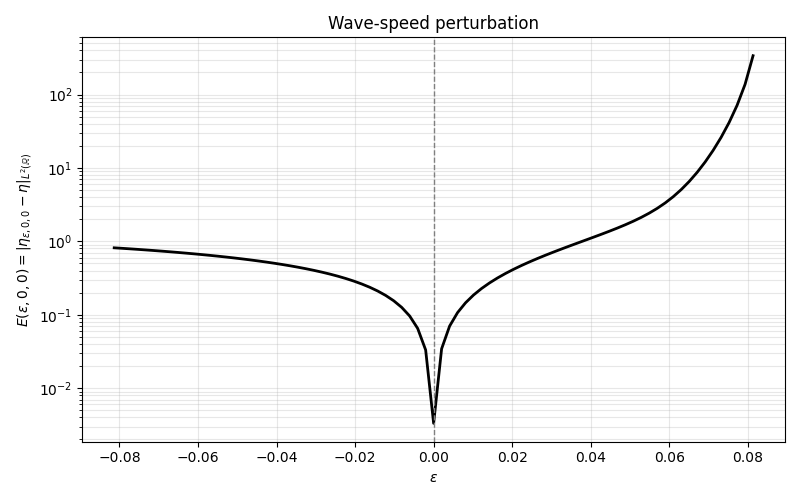}};
\end{tikzpicture}
\caption{Reconstruction error due to wave-speed perturbation}
\label{wavespeederror}
\end{figure}

\subsection{Pressure Perturbation}
In this subsection, we consider the case in which only the pressure trace is perturbed, while the wave-speed parameter and the depth remain fixed. More precisely, we set $
\varepsilon=0,
\gamma=0,$
and perturb only the bed pressure through $
\mathfrak{p}_\delta(q)=\mathfrak{p}(q)+\delta(q)$. The perturbation is different from the one used in \cite{bleile2016}. The perturbed pressure trace reads
\[
\mathfrak{p}_\delta(q)=\mathfrak{p}(q)+a e^{-q^2/(2\theta^2)}.
\]
Since \(\varepsilon=0\) and \(\gamma=0\), the wave-speed parameter remains fixed at \(\kappa=\frac14\), and the depth also remains fixed at \(d=1\).
Moreover, the corresponding reconstructed surface profile here is given by 
\[
\eta_{0,\delta,0}(q)
=
\mathcal F^{-1}
\left[
\frac{\sinh(kd)}{k}\,\widehat g_{0,\delta}(k)
\right](q).
\]

To ensure that the reconstruction formula is well-defined in this scenario, the term inside the square root in \(g_{0,\delta}\) must remain real, namely
\[
1-2\kappa\bigl(\mathfrak{p}(q)+\delta(q)\bigr)>0
\qquad\text{for all }q\in\mathbb R.
\]
However, since $
\kappa=\frac14,$
this condition translates to
\[
1-\frac12\bigl(\mathfrak{p}(q)+\delta(q)\bigr)>0.
\]
Thus, the amplitude \(a\) must be chosen sufficiently small so that the perturbed pressure does not force the radicand to vanish or become negative. It is worth adding that in our numerics, we set $\theta=2$ to remove one degree of freedom and simplify the computation. We obtain the following graph explaining the connection between the pressure trace perturbation, determined by the amplitude $a$, and the reconstruction  error. As seen in Figure~\ref{amplitudeerror}, when the measurement error comes only from the pressure trace, the surface reconstruction error increases in the $L^2_q$ sense as the amplitude of $\delta$ increases. 
\begin{figure}[H]
\begin{center}
    \begin{tikzpicture}
    \node[inner sep=0pt] at (5,0) {\includegraphics[width=9.5cm]{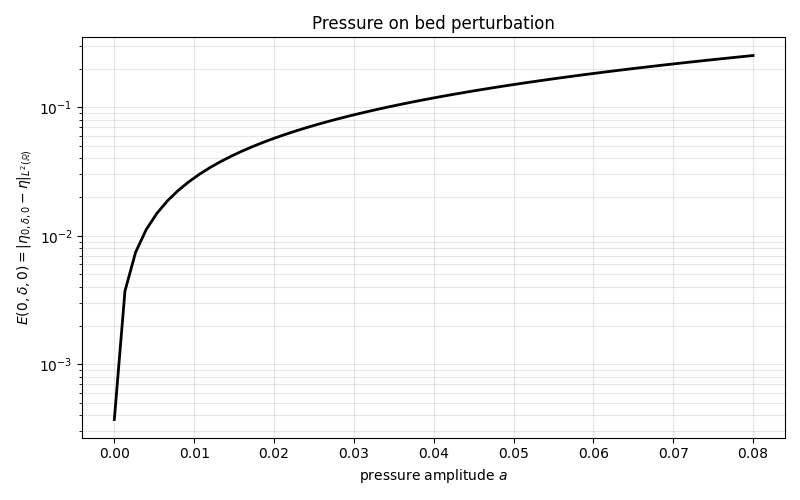}};
\end{tikzpicture}
\caption{Reconstruction error due to bed-pressure perturbation}
\label{amplitudeerror}
\end{center}
\end{figure}

\subsection{Depth Perturbation}
Now, we consider the case in which only the fluid depth is perturbed, while the wave-speed parameter and the pressure trace remain fixed. More precisely, we set
$\varepsilon=0,
\delta=0,$
and allow only the depth to vary through $d_\gamma=d+\gamma$.
Therefore, the perturbed recovered profile takes the form
\[
\eta_{0,0,\gamma}(q)
=
\mathcal F^{-1}
\left[
\frac{\sinh(k(d+\gamma))}{k}\,\widehat g_{0,0}(k)
\right](q).
\]
Thus, in the depth-only perturbation experiment, the sole source of error comes from the change in the Fourier multiplier from $\sinh(kd)/k$ to $\sinh(k(d+\gamma))/k$.

Notice that here, the square-root factor in \(g_{0,0}\) is unchanged. Hence, the admissibility condition associated with the radicand is automatically satisfied once the pressure trace on the bed is admissible. The only additional requirement is for the perturbed depth to remain physically meaningful, namely $d+\gamma>0$.
Since we fix \(d=1\), we then require $
\gamma>-1$.
Thus, in the numerical implementation, one must choose \(\gamma\) in a range that preserves positive depth. We obtain Figure~\ref{dpth} which relates the perturbation $\gamma$ with the surface reconstruction error.
\begin{figure}[H]
\begin{center}
\begin{tikzpicture}
    \node[inner sep=0pt] at (5,0) {\includegraphics[width=9.5cm]{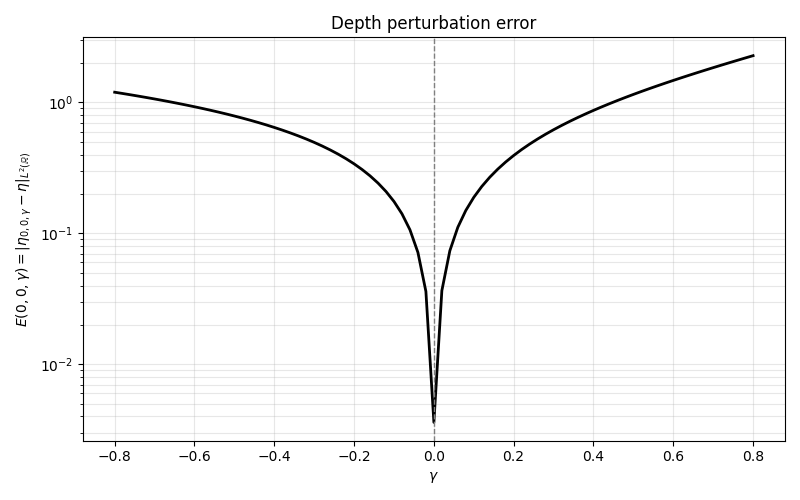}};
\end{tikzpicture}
\caption{Reconstruction error due to depth perturbation}
\label{dpth}
\end{center}
\end{figure}

\newpage
\begin{center}
\bibliographystyle{alpha}
\bibliography{Ref.bib}
\end{center}

\end{document}